\documentclass{amsart}
\usepackage{graphicx}
\usepackage{hyperref}
\usepackage{graphicx,amssymb,amstext,amsthm,url}
\newcommand{\Deltab}{\bigtriangleup}

\begin{document}

\newtheorem{theorem}{Theorem}  [section]
\newtheorem{proposition}[theorem]{Proposition}
\newtheorem{lemma}[theorem]{Lemma}
\newtheorem{clm}[theorem]{Claim}
\newtheorem{corollary}[theorem]{Corollary}
\newtheorem{conj}[theorem]{Conjecture}
\newtheorem{conjecture}[theorem]{Conjecture}

\title{Maps and $\Delta$-matroids revisited}

\author{R\'emi Cocou Avohou}
\address[R.C.A.]{Einstein Institute of Mathematics, The Hebrew University of Jerusalem, Giv'at Ram, Jerusalem, 91904, Israel, \&
ICMPA-UNESCO Chair, 072BP50, Cotonou, Rep. of Benin,  \& Ecole Normale Superieure, B.P 72, Natitingou, Benin}
\email{avohou.r.cocou@mail.huji.ac.il}

\author{Brigitte Servatius}
\address[B.S.]{Mathematical Sciences, Worcester Polytechnic Institute, Worcester MA 01609-
2280}
\email{bservat@wpi.edu }

\author{Herman Servatius}
\address[B.S.]{Mathematical Sciences, Worcester Polytechnic Institute, Worcester MA 01609-
2280}
\email{hservat@wpi.edu }

\maketitle

\begin{abstract} Using Tutte's combinatorial definition of a map we define a $\Delta$-matroid
purely combinatorially and show that it is identical to Bouchet's topological definition.
\end{abstract}

\section{Matroids and $\Delta$-matroids}
A {\em matroid} $M$ is a finite set $E$ and a collection $\mathcal{B}$ of subsets of $E$
satisfying the condition that if
\begin{enumerate}\renewcommand{\theenumi}{MB}
\item
If $B_1$ and $B_2$ are in $\mathcal{B}$ and $x \in B_1 \setminus B_2$ then there
exists a $y \in  B_2 \setminus B_1$ such that $(B_1 \cup \{ y \} ) \setminus \{ x \} = B_1 \Deltab \{x,y\} \in \mathcal{B}$
\label{BaseAxiom}
\end{enumerate}
Axiom (MB) is called the
{\em basis exchange axiom}. Sets in $\mathcal{B}$ are called {\em bases} of $M$.

Replacing the set difference in Axiom~(MB) by the symmetric difference we obtain the symmetric exchange axiom ($\Delta$F) used by
Bouchet~\cite{greedydelt} to define $\Delta$-matroids.

A {\em $\Delta$-matroid} $D$ is a finite set $E$ and a collection $\mathcal{F}$ of subsets of $E$
satisfying the condition that if
\begin{enumerate}\renewcommand{\theenumi}{$\Delta$F}
\item If $F_1$ and $F_2$ are in $\mathcal{F}$ and $x \in F_1 \Deltab  F_2$ then there
exists a $y \in  F_2 \Deltab F_1$ such that $F_1 \Deltab \{x,y\} \in \mathcal{F}$.
\label{FeaseAxiom}
\end{enumerate}
Axiom ($\Delta$F) is called the {\em symmetric exchange axiom} and the sets in $\mathcal{F}$ are called the {\em feasible
sets} of $D$.  It is important to note that $y$ may equal $x$, so
$|F_1 \Deltab \{x,y\}| - |F_1| \in \{0, \pm 1 , \pm 2\}$.

 There are two obvious matroids associated with every $\Delta$-matroid;
$M_u$, the {\em upper matroid}, whose bases are the feasible sets with largest cardinality, and
$M_l$, the {\em lower matroid}, whose bases are the feasible sets with least cardinality,~\cite{Bouchet}.

\section*{Delta Matroids and maps on surfaces}
In~\cite{Bouchet}, Bouchet associates a $\Delta-$matroid to any map.
A map is a cellular embedding of a graph $G$ into a compact surface, and, for the $\Delta-$matroid he defined, the lower matroid is the
cycle matroid of $G$, and the upper matroid is the dual of the cycle matroid of the geometric dual, $G^*$, of $G$ in the surface.
For more information about maps see~\cite{Tutte,godsil2001algebraic}.
In this section we would like to reformulate the connection between maps and $\Delta$-matroids in such a way as to clarify both
the geometry and the combinatorics.

Bouchet defined a {\em base} $B$ of a map  as a selection of edges from the cellularly embedded graph, $B \subseteq E$, such that,
after deleting all the edges of $B$
and all the dual edges of $E \setminus B$, together with their endpoints, the resulting non-compact surface is connected.
To perform this operation, it is convenient to use the barycentric subdivision of the map, whose one-skeleton contains both the graph and the dual-graph,
with the edges of each subdivided in two, see Figure~\ref{mappycrappy00ab}a and b.
\begin{figure}[htb]
\centering
a)\includegraphics[scale=.5]{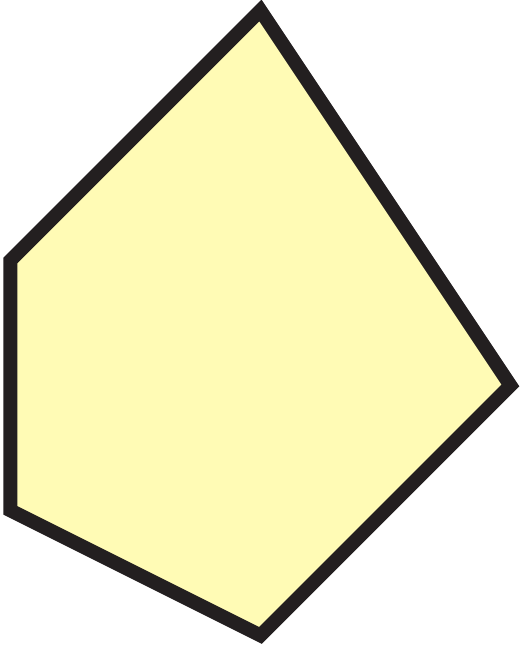}  b)\includegraphics[scale=.5]{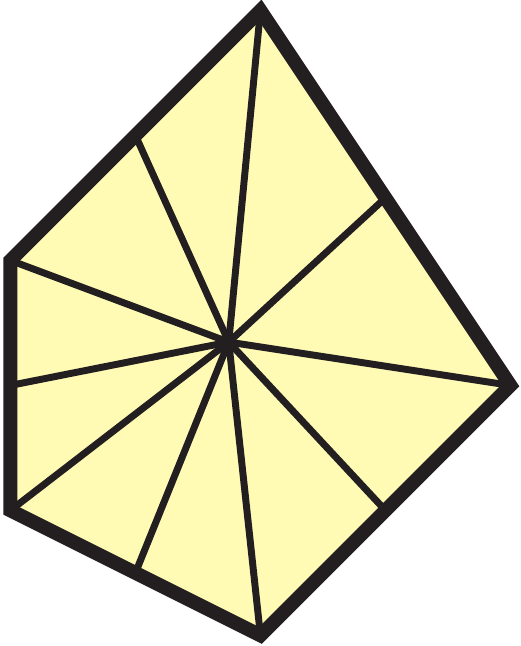}
 c)\includegraphics[scale=.5]{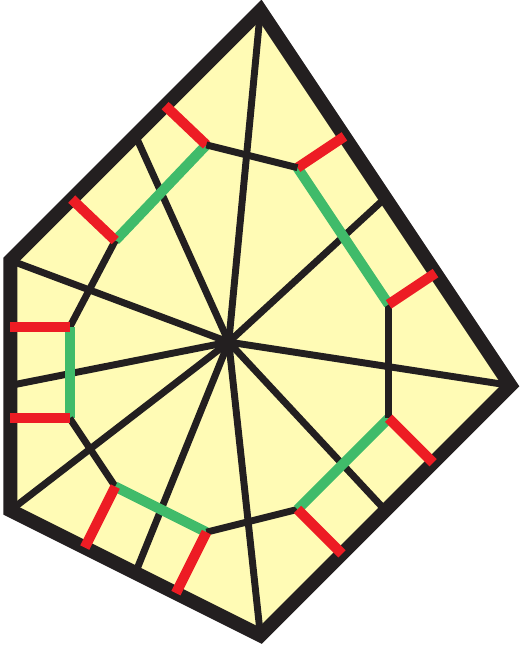}  c)\includegraphics[scale=.5]{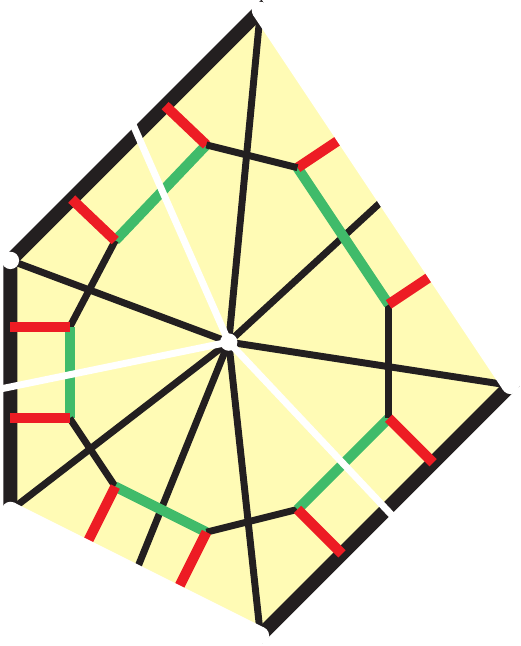}
 \caption{a) A cell of a map, b) its barycentric subdivision, c) the map graph, d) deleting an edge/dual-edge selection.\label{mappycrappy00ab}}
\end{figure}
The {\em map graph} is the geometric dual of the barycentric subdivision, Figure~\ref{mappycrappy00ab}c, where
the edges are colored green, red, and black depending on whether they are parallel to one of the original edges, cross one, or neither.
 Suppose, as Bouchet did, we delete, for each edge, either the edge
or its dual, together with their endpoints, as realized  in the barycentric subdivision. 
If it should happen that some vertex or dual vertex of the map is not deleted,
then it is an interior point of the the deleted surface, and we may
puncture the surface there without affecting the connectivity.
Then, expanding the holes
at the vertices and dual vertices, there is a deformation of the punctured surface which respects all the edges and dual edges,
\begin{figure}[htb]
\centering
a)\includegraphics[scale=.5]{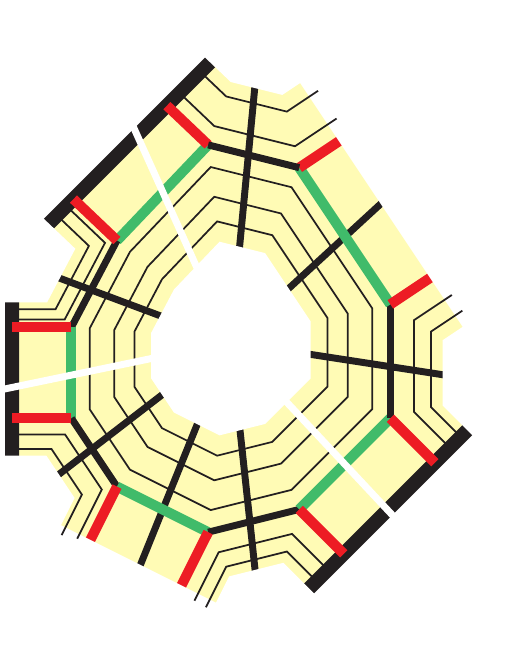} \qquad b)\includegraphics[scale=.5]{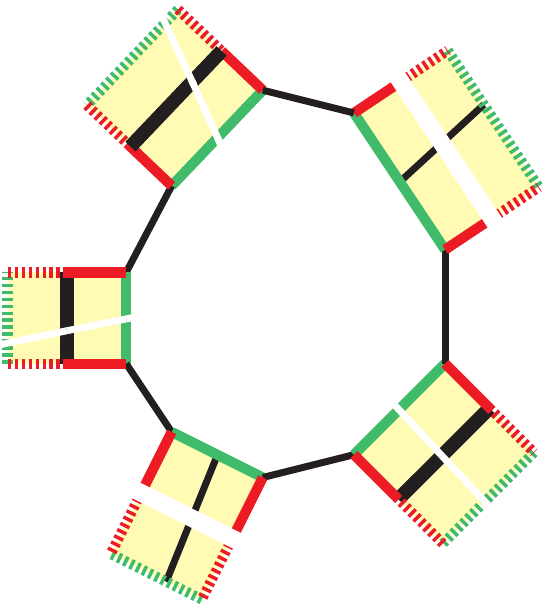} \qquad c)\includegraphics[scale=.5]{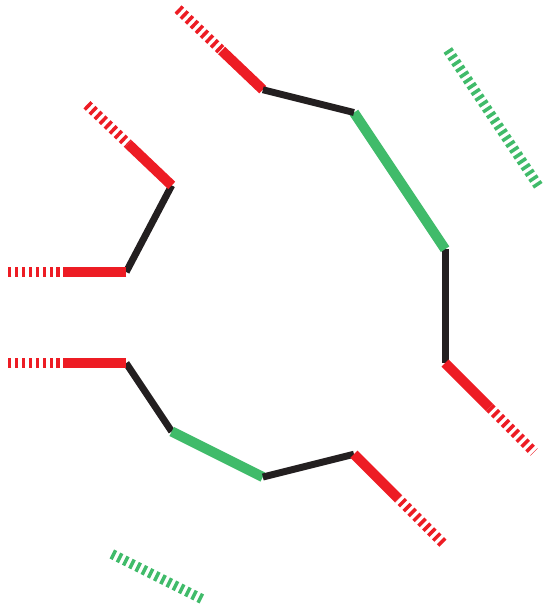}
 \caption{Deforming away from the vertex and dual vertex holes.\label{mappycrappy456}}
\end{figure}
so, in particular, respecting the deleted edges.
This deformation can continue, expanding the holes until
all that is left is the set of black edges of the map graph and the green-red quadrilaterals, each of which has been cut in half,
either leaving the green edge pair intact, or the red edge pair.
Each of these cut quadrilaterals can and be deformed, expanding the cut, onto the surviving color pair, leaving the map graph with one
 color pair deleted from each quadrilateral, green for those in $B$, and red for the others.
This is a $2$-regular subgraph of the map graph, and contains all the black edges.   By the deformation, the
surface with the edges and dual edges deleted is connected if and only if the corresponding $2$-regular subgraph of the map graph is connected as
a topological space, which is true if and only if that $2$-regular subgraph is a Hamiltonian cycle.

Bouchet went on to show that the sets $B$ formed the feasible sets of
a $\Delta$-matroid on $E$,
using Eulerian splitters.
Using the map graph, we may establish this simply and directly.

\section{Combinatorial Maps and $\Delta$-matroids}
Tutte, in the introduction to his paper {\em What is a map?}~\cite{Tutte} remarks
\begin{quote}
 Maps are usually presented as cellular dissections of topologically defined surfaces. But some combinatorialists, holding that maps
 are combinatorial in nature, have suggested purely combinatorial axioms for map theory,
 so that that branch of combinatorics can be developed without appealing to point-set topology.
\end{quote}
Tutte's idea is that each edge of a map is associated with four flags, corresponding to
the triangles in the barycentric subdivision. Each flag has three vertices: one corresponding to a vertex of the embedded graph,
one corresponding to an edge, and one corresponding to a face of the map.
The map can be uniquely described in terms
of three perfect matchings. Two flags are matched if they share a vertex of the same kind.  Faces, Euler characteristic,
and orientability  can be treated combinatorially without appealing to topology. We now recall Tutte's
axiomatic approach as presented in~\cite{godsil2001algebraic}.


Let $\Gamma$ be a connected graph whose edges are partitioned into three classes
$R$, $G$, and $B$ which we color respectively red, green, and black.  $\Gamma$ is called {\em map graph}
or a {\em combinatorial map} if the following conditions are satisfied:
\begin{enumerate}
\item Each color class is a perfect matching
\item $R \cup G$ is a union of $4$-cycles
\item $\Gamma$ is connected
\end{enumerate}
The graph $\Gamma$ is 3-regular and edge $2$-connected.
$\Gamma$ may have parallel edges, although necessarily not red/green.
$\Gamma$ contains $2$-regular subgraphs which use all the black edges
of $\Gamma$,
which we call {\em fully black} $2$-regular subgraphs;  $R\cup B$ and $G \cup B$ are
examples, and there always exists a fully black Hamiltonian cycle.
To see this, first note that a fully black $2$-regular subgraph cannot
contain any incident green and red edges, so
every red/green quadrilateral intersects a fully black $2$-regular
subgraph in either two red, or two green edges.
Now consider a fully black $2$-regular subgraph of $\Gamma$ with the
fewest connected components. If there is not a single component, then there
is a green/red quadrilateral which intersects the subgraph in, say, two red edges
which belong to two different components, and swapping red and green
on that quadrilateral reduces the number of components of the subgraph, violating minimality.

\begin{theorem}
 Given a combinatorial map $\Gamma(R,G,B)$, let $E$ be the set of quadrilaterals of $R \cup G$, and let $\mathcal{F}$  be the collection
 of subsets of $E$ corresponding to the pairs of green edges in a fully black Hamilton cycle in $\Gamma$. Then $(\mathcal{F}, E)$ is
 a $\Delta$-matroid.
\end{theorem}

\begin{proof}
 We have to show the symmetric exchange property.
Let $F_C$ and $F_{C'}$ be sets of quadrilaterals corresponding to fully black Hamiltonian cycles $C$ and $C'$.
Let $q \in F_C \Deltab F_C'$, so the edges of quadrilateral $q$ are differently colored in $C$ and $C'$, say red and green.
There are two cases, either replacing in $q$ the red edges in $C$ with the green of $C'$ results in two components or one.
See Figure~\ref{deltydeltfig01a}.
\begin{figure}[htb]
\centering
\includegraphics{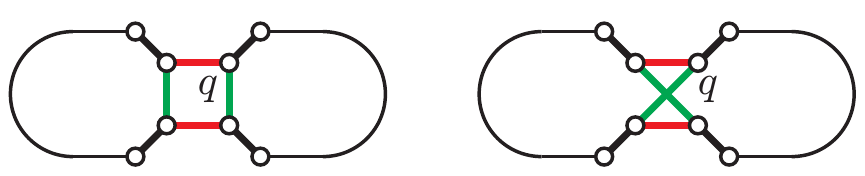}
\caption{\label{deltydeltfig01a}}
\end{figure}
If it results in just one component, then take $q' = q$, and $F_c \Deltab \{q,q'\} = F_c \Deltab \{q\}$ is the set of
red quadrilaterals of a fully black Hamiltonian cycle, and hence feasible, as required.

Otherwise, if there are two components, the Hamiltonian cycle of $C'$ contains a non-black edge, say green, of a quadrilateral
$q_1$, connecting
those two components, and necessarily both red edges of $q_1$ are in $C$ and both green edges of $q_1$ connect the components,
and $q' \in C \Deltab C'$.
See Figure~\ref{deltydeltfig02a}.
\begin{figure}[htb]
\centering
\includegraphics{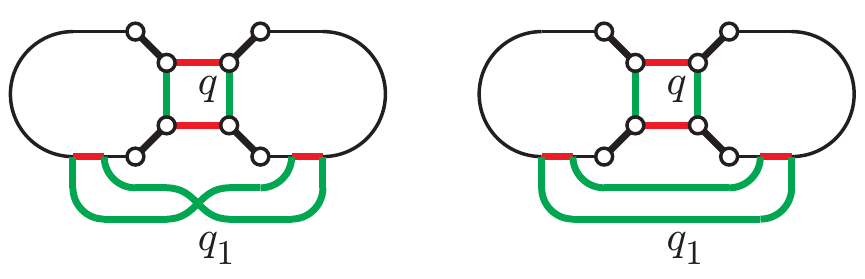}
\caption{\label{deltydeltfig02a}}
\end{figure}
Regardless of how the green edges of $q_1$ are placed, swapping the edges of both $q$ and $q'$ in $C$ yields
a new fully black Hamiltonian cycle, so the set $Q \Deltab \{q,q_1\}$ is feasible, as required.
\end{proof}


Since $R$, $G$ and $B$ are perfect matchings, the union of any two them induces a set of disjoint cycles.
Let $V$ be the set of cycles of $R \cup B$,
$E$ be the set of cycles  of $R \cup G$,
and $V^*$ be the set of cycles of $G \cup B$.
There is a graph $(V,E)$ where incidence is defined between a red-black cycle and a red-green cycle if they share and edge, and, similarly,
there is a graph $(V^*,E)$ where incidence is defined between a green-black cycle and a red-green cycle if they share and edge.
%
%
%
We say that $\Gamma$ encodes the graph $(V,E)$ and its geometric dual $(V^*,E)$.

\begin{theorem}
Let  $\Gamma(R,G,B)$ be a combinatorial map and  $D_{\Gamma} =(\mathfrak{F},E)$ its associated $\Delta$-matroid.
Then the lower matroid of $D_{\Gamma}$ is the cycle matroid of the graph $(V,E)$
and the upper matroid of $D_{\Gamma}$ is the cocycle matroid of the graph $(V^*,E)$.
\end{theorem}

\begin{proof}
Given $\Gamma(R,G,B)$, recall that the feasible sets of $D$ consist of $RG$ quadrilaterals whose $R$ edges are contained in a
fully black Hamilton cycle of $\Gamma$. Any fully black Hamilton cycle $C$ of $\Gamma$ must contain the red edges corresponding to
a spanning tree of $(V,E)$ as well as the green edges corresponding to a spanning tree of $(V^*, E)$. So the minimal number of red edges in $C$ is
$2(|V|-1)$, while the maximal number is $2(|E| - |V^*| + 1)$. The edge sets of the spanning trees of $(V,E)$ are the bases of its cycle matroid,
while the complements of edge sets of spanning trees in $(V^*, E)$ are the bases of the cocycle matroid of $(V^*, E)$.
\end{proof}

Note that the difference in rank of the upper and lower matroid of $(\mathfrak{F},E)$ is given by $(|E| - |V^*| + 1)- (|V|-1)= 2- \chi$, where
$\chi$ is the Euler characteristic. Notice also, that if $\Gamma$ is bipartite, all feasible sets of $D_{\Gamma} =(\mathfrak{F},E)$ must have
the same parity, since exchanging a red and green pair of edges always disconnects a Hamilton cycle of a bipartite $\Gamma$.

Examples of combinatorial maps together with the underlying graph and geometric dual are provided in 
Figures~\ref{deltamapexample01},~\ref{deltamapexample02},~\ref{deltamapexample03}.

\begin{figure}[htb]
\centering
\includegraphics{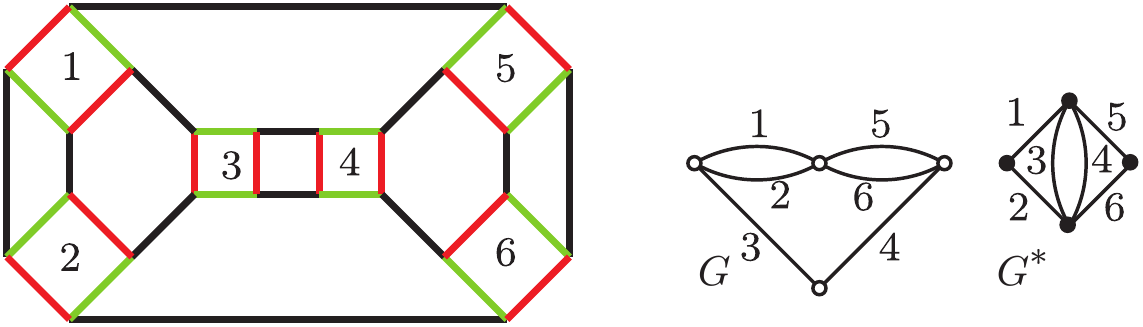}
\caption{\label{deltamapexample01}}
\end{figure}

The $\Delta$-matroid associated to the map of Figure~\ref{deltamapexample01} has feasible sets
$$\mathfrak{F} = \{ \{1,3,4\}, \{1,3,5\}, \{1,3,6\}, \{1,4,5\}, \{1,4,6\}, 
\{2,3,4\},$$
$$ \{2,3,5\}, \{2,3,6\}, \{2,4,5\},
\{2,4,6\}, \{3,4,5\}, \{3,4,6\} \}$$. Note that $\mathfrak{F}$ is the set of spanning trees of $G$ and at the same
time the set of co-trees of $G^*$, so all feasible sets have the same size and upper and lower matroid are identical.

\begin{figure}[htb]
\centering
\includegraphics{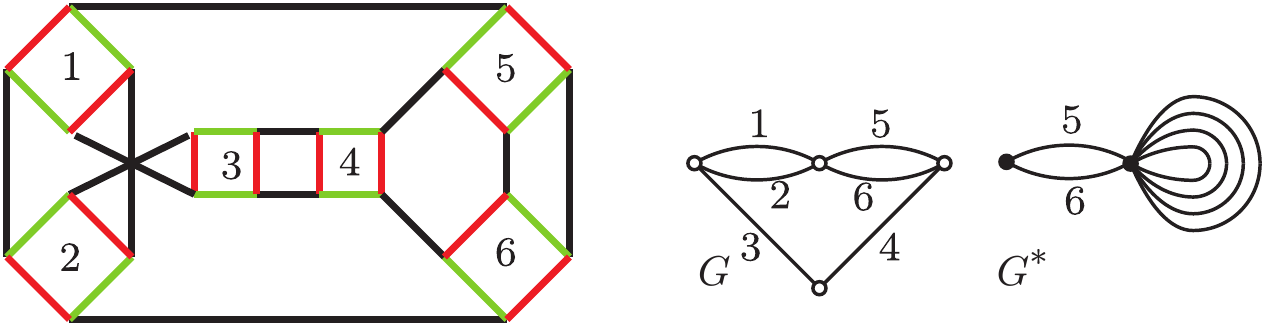}
\caption{\label{deltamapexample02}}
\end{figure}

The $\Delta$-matroid associated to the map of Figure~\ref{deltamapexample02} has feasible sets all
the sets in $\mathfrak{F}$ together with the two additional sets $\{1,2,3,4,5\}$ and $\{1,2,3,4,6\}$. 
The lower matroid is again the cycle matroid of $G$, but the upper matroid is the co-cycle matroid 
of $G^*$, which has rank~5 and contains exactly one cycle, namely $\{5,6\}$, which is a minimal cutset
of $G^*$ and also a cycle in $G$.

\begin{figure}[htb]
\centering
\includegraphics{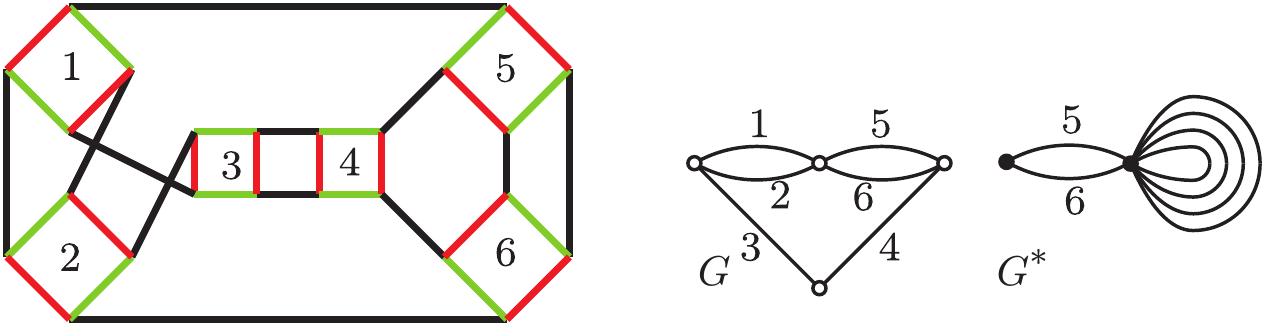}
\caption{\label{deltamapexample03}}
\end{figure}

The $\Delta$-matroid associated to the map of Figure~\ref{deltamapexample03}
has, in addition to the feasible sets of the previous example the feasible set $\{1,2,3,4\}$, whose parity is even,
while the parity of all other feasible sets is odd, so this map is not orientable.

As is clear from these examples, the map cannot, in general be recovered from the  $\Delta$-matroid information,
since the upper or lower matroid do not even determine the graph. Non-isomorphic graphs may have identical cycle-and
co-cycle matroids. It is easy to check that $\mathfrak{F}$ is also a list of spanning trees for the graph $G'$, but
$G$ is not isomorphic to $G'$.

However, if both $G$ and $G^*$ are 3-connected, then the map is uniquely recoverable from the $\Delta$-matroid information.

\begin{theorem}\label{DtoMThm}
Let $D$ be the $\Delta$-matroid of a map $M$ with 2-connected upper- and lower matroid.
Then $M$ is determined by $D$. 
\end{theorem}

\begin{proof}
By Whitney's theorem~\cite{Whitney}, upper and lower matroid uniquely determine $G$ and $G^*$. 
To recover $M$ from $D$, we need to specify a rotation system for each vertex $v$ of $G$. To determine
if two edges  $e$ and $f$ with endpoint $v$ follow each other in the rotation about $v$, it is enough 
to check if $e$ and $f$ are both incident in $G^*$, since the vertex co-cycles of $G^*$ correspond to the
facial cycles of the embedded $G$. Now re-construct the map graph.
\end{proof}

\begin{figure}[htb]
\centering
\includegraphics{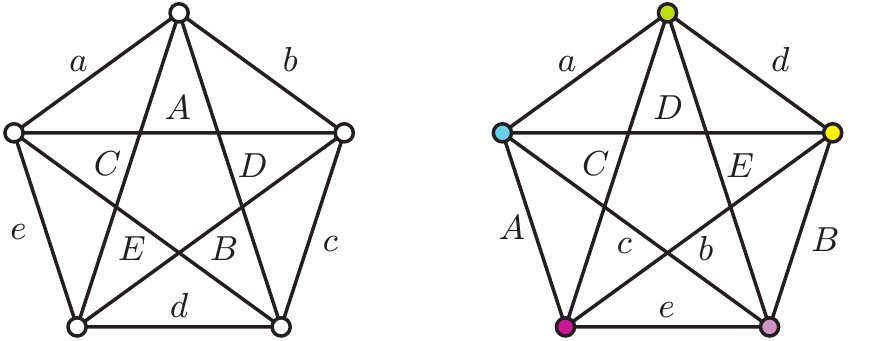}
\caption{\label{k5torusmapagainfig01}}
\end{figure}

\begin{figure}[htb]
\centering
\includegraphics{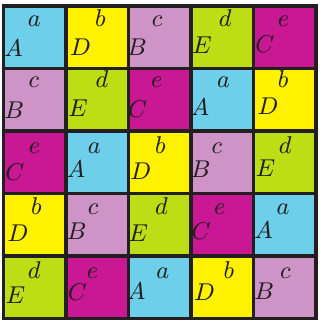}
\caption{\label{k5torusmapagainfig02}}
\end{figure}

For example the lower matroid could be the cycle matroid of $K_5$, while the upper matroid
is the co-cycle matroid of $K_5$ as well, so this matroid information gives us the graphs
$G$ and $G^*$ depicted in Figures~\ref{k5torusmapagainfig01}. By the method in the proof of Theorem~\ref{DtoMThm}
the map $M$ is easily recovered to be as in Figure~\ref{k5torusmapagainfig02}.


\section{Another $\Delta$-matroid from a map}

If the objective is to define a natural $\Delta$-matroid from a combinatorial map,
the requirement that the subgraph of the map graph be Hamiltonian
can be weakened provided that some connectivity is required.
Again, let $\Gamma$ be a map graph with edge set $R \cup G \cup B$, with red edges $R$, green edges $G$ and black edges $B$.

\begin{theorem}\label{theo:twoval}
Let $K$ be a fully black $2$-valent subgraph of $\Gamma$ with the property that $K \cup R$ and $K \cup G$ are both connected.
Then the set $F_K$ of quadrilaterals in which red is
selected in $K$ form the feasible sets of a $\Delta$-matroid.
\end{theorem}

\begin{proof}
We have to show the symmetric exchange property.
Let $F_K$ and $F_{K'}$ be sets of red quadrilaterals corresponding
to fully black $2$-valent subgraphs $K$ and $K'$, both of which can be
connected by adding edges of one color only.
Let $q \in F_K \Deltab F_K'$, so the edges of quadrilateral $q$ are differently colored in $K$ and $K'$, say red and green respectively.
If the red edges of $q$ belong to two different cycles of $K$, then swapping red for green in $q$ merges the two cycles, then
we may take $q' = q$ and the $F_K \Deltab \{q,q'\}$ will be connected by the same collections of red, respectively green edges as $F_k$.

So we may assume that the red edges of $q$ belong to the same component $K$.  If swapping red for green in $q$
does not split the component of $K$ they belong to, see the right side of Figure~\ref{deltydeltfig01a}, then just as before,
take $q' = q$.  So we may assume that the red edges of $q$ belong to the same component of $K$, and swapping them
for green splits that component, see the left side of Figure~\ref{deltydeltfig01a}.
Let the red edges of $q$ be denoted by $q_r$ and the green edges of $q$ be denoted by $q_g$.
Clearly $(K-q_r + q_g) + R$ is connected since $K  + R  \subseteq (K -q_r + q_g) + R $ and is connected.
The issue is that $(K-q_r + q_g) + G = K + G - q_r$ may have two components.  If it has just one, again, take $q' = q$ and we are done.
We know that $K' + G$ is connected, so $K'$ must have a red edge of some quadrilateral $q'$ that connects the two components of
$(K-q_r + q_g) + G$, so $q' \not \in F_K$ and $q' \in F_{K'}$, that is $q' \in F_K \Deltab F_{K'}$.
$(K - q_r + q_r') + G$ is connected and we already know
$(K - q_r + q_r') + R$ is connected, so the collections $F_K$ are the feasible sets of a $\Delta$-matroid.
\end{proof}

Let $D_{\Gamma}$ be the $\Delta$-matroid as in Theorem \ref{DtoMThm}, with feasible sets the pairs of red edges in a fully black Hamilton cycle, and $D_K$ be the $\Delta$-matroid as in Theorem \ref{theo:twoval}, with feasible sets the pairs of red edges in a fully black $2$-valent subgraph $K$ such that $K$ becomes
connected by addition of red edges only as well as by addition of only green edges. $D_\Gamma$ and $D_K$ are different. For example for the unitary map there are two quadrilaterals, $\{q, q'\}$ and the feasible sets in the first sense are $\{\emptyset, \{q, q'\}\}$, whereas in the second sense are all subsets.
%
For unitary maps the connectivity issue
here is void since the $R + B$ and $G + B$ are both connected.
The upper and lower matroid for both $\Delta$-matroids we defined are the same.
However, the Hamiltonian requirement
encodes the orientability of the map, by the fact that all feasible sets have the same parity in the orientable case and are
of both even and odd cardinality if $\Gamma$ is not bipartite, while the second approach does not distinguish between the two.


\vspace{0.5cm}


\end{document}